\documentclass{scrartcl}
\usepackage[british]{babel}
\usepackage[textwidth=150mm,top=2cm]{geometry}
\usepackage{dsfont, amsthm, mathtools,  enumerate,amsmath,amsrefs,libertine}

\newtheorem*{theorem}{Theorem}
\newtheorem*{remark}{Remark}
\newtheorem*{corollary}{Corollary}
\newtheorem*{acknowledgement}{Acknowledgement}
\newtheorem*{AddedInProof}{Added in Proof}

\input cyracc.def
\font\tencyr=wncyr7
\def\cyr{\tencyr\cyracc}

\begin{document}

\title{The ``Wrong Minimal Surface Equation" does not have the Bernstein property}
\author{Peter \textsc{Lewintan}}\date{}
\author{Peter \textsc{Lewintan}\footnote{peter.lewintan@uni-due.de, University of Duisburg-Essen, Germany}}
\date{May 9, 2011}
\maketitle

\begin{abstract}
\noindent
A celebrated result of S. \textsc{Bernstein} \cite{Bernstein} states that every solution of the minimal surface equation over the entire plane $\mathds R^2$ has to be an affine linear function. Since the paper of \textsc{Bernstein} appeared in 1927, many different proofs and generalizations of this beautiful theorem were given, namely to higher dimensions and to more general equations, for a careful account we refer to the paper by \textsc{Simon} \cite{Simon2} and to the monograph by \textsc{Dierkes-Hildebrandt-Tromba} \cite[chap. 3]{DHT}.

In his paper \cite{Simon} \textsc{Simon} posed the question whether the equation
\begin{equation}
(1+{u_x}^2)u_{xx}+2u_x u_y u_{xy}+ (1+{u_y}^2)u_{yy} = 0
\label{toll}
\end{equation}
has the Bernstein property i.e. whether every $C^2$-solution defined on all of $\mathds R^2$ necessarily has to be affine.

We here show by a very simple argument that this is not the case.
\end{abstract}

\textit{Keywords:}
Bernstein property, wrong minimal surface equation, entire non-linear solutions\\[2pt]
\textit{AMS 2000 MSC:} 35A01, 35B08, 35B65, 35D05, 35J15\\[2pt]

To start with we consider $u\in C^2(\mathds R^2)$ to be a solution of the elliptic equation \eqref{toll} with $u_{xy} \equiv 0$ in the whole of $\mathds R^2$. Then $u$ has the form
\begin{align*}
u(x,y) = h(x)+g(y), &\ \text{with} \ g,h \in C^2(\mathds R) \\
 \text{and the equation \eqref{toll} becomes:} \ \  &(1+(h'(x))^2)h''(x) + (1+(g'(y))^2)g''(y) = 0.
\end{align*}
We put $(1+(h'(x))^2)h''(x) = c$ and hence $(1+(g'(y))^2)g''(y) = -c$, $c\in \mathds R$, and choose $c = 1$. (To get the linear solutions take $c=0$.)

By separation of variables we solve the equation  \ $(1+f^2)f' = 1$  with  $f=h'(x)$ \ and obtain \ \ $(1+f^2)df = dx$,\ \ or \ $f + \frac{\displaystyle f^3}{\displaystyle 3} = x$.
By Cardano's formulae:
\begin{equation*}
h'(x) = f(x) = \frac{1}{\sqrt[3]{2}}\left(\sqrt[3]{\sqrt{9x^2+4}+3x} -\sqrt[3]{\sqrt{9x^2+4}-3x} \right).
\end{equation*}
An integration yields:
\begin{equation*}\begin{split}
 h(x) = &\frac{-1}{\sqrt[3]{1024}}\left\{9x\left(\sqrt[3]{\sqrt{9x^2+4}-3x} -\sqrt[3]{\sqrt{9x^2+4}+3x} \right)\right.\\
 &\qquad\qquad\ \left.+\sqrt{9x^2+4}\left(\sqrt[3]{\sqrt{9x^2+4}-3x} +\sqrt[3]{\sqrt{9x^2+4}+3x} \right)\right\}.
 \end{split}
\end{equation*}
Similarly, we get
\begin{equation*}\begin{split}
g(y) =  &\frac{1}{\sqrt[3]{1024}}\left\{9y\left(\sqrt[3]{\sqrt{9y^2+4}-3y} -\sqrt[3]{\sqrt{9y^2+4}+3y} \right)\right.\\
 &\qquad\qquad\ \left.+\sqrt{9y^2+4}\left(\sqrt[3]{\sqrt{9y^2+4}-3y} +\sqrt[3]{\sqrt{9y^2+4}+3y} \right)\right\} = -h(y).
 \end{split}
\end{equation*}

Thus, the non-linear $C^2-$function
\begin{equation*}
u(x,y) = h(x)-h(y)
\end{equation*}
solves \eqref{toll} in the whole plane $\mathds R^2$.
\vspace*{-4em}
\begin{flushright}
\includegraphics[width=8cm]{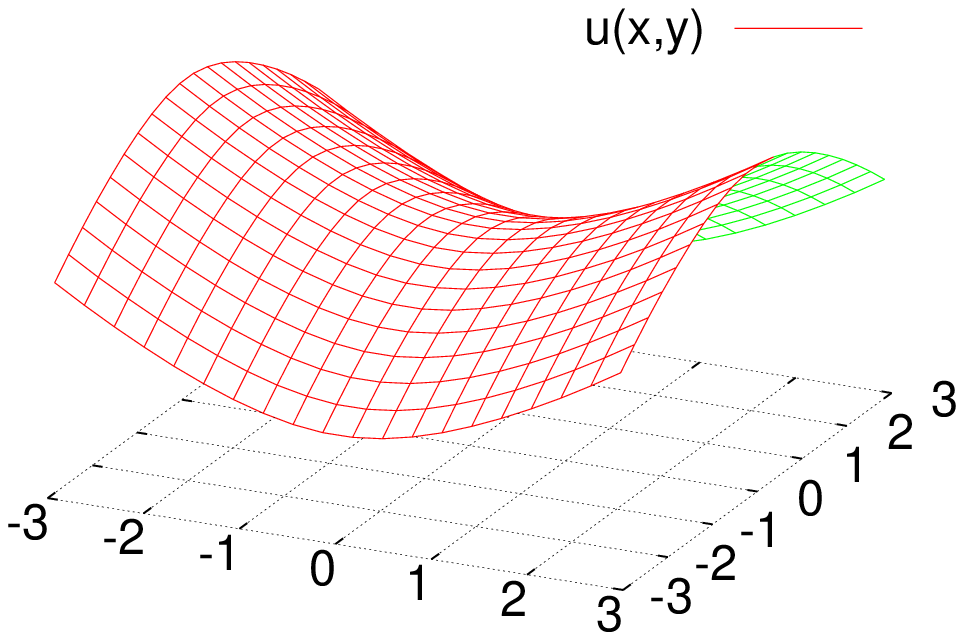}
\end{flushright}
\vspace*{-2.5em}

\begin{remark}
The above $u$ solves also the elliptic equation
\begin{equation*}
(1+{u_x}^2)u_{xx}-2u_x u_y u_{xy}+ (1+{u_y}^2)u_{yy} = 0.
\end{equation*}
\end{remark}

More generally, we have the following
\begin{theorem}
Let  $F_i\in C^1(\mathds R, \mathds R)$ be bijective with positive derivative $F_i' = f_i > 0$ for $i=1,2$. Then the equation
\begin{equation} f_1(u_x)\cdot u_{xx}+ 2B\cdot u_{xy}+ f_2(u_y)\cdot u_{yy} = 0,\label{eigenThm} \end{equation}
with an arbitary $B$ (depending on $x,y,u,u_x,u_y,u_{xx},u_{xy},u_{yy}$) has non-linear entire $C^2-$solutions in $\mathds R^2$, i.e. \eqref{eigenThm} does not have the Bernstein property. \\
(For the ellipticity of \eqref{eigenThm} assume $|B|< \sqrt{f_1(u_x)f_2(u_y)}$.)
\end{theorem}

\begin{proof}
We proceed analogously as above: To the end we construct a $C^2-$solution $u$ with $u_{xy}\equiv0$, i.e. $u(x,y) = h(x)+g(y), \ \text{with} \ g,h \in C^2(\mathds R)$. Thus, our equation \eqref{eigenThm} becomes: \begin{equation*} f_1(h'(x))h''(x)+f_2(g'(y))g''(y)=0.\end{equation*}
Put $f_1(h'(x))h''(x) = c$ and $f_2(g'(y))g''(y) = -c$, with an arbitrary constant $c\in \mathds R$.\\
For the linear solutions take $c=0$. Since we are interested in non-linear ones, let us choose $c=1$:\\
By separation of variables we get: \begin{equation*} F_1(h'(x)) = x \qquad \text{  and  } \qquad F_2(g'(y)) = -y.\end{equation*}

Since $F_i$ ($i=1,2$) is bijective in the whole of $\mathds R$, a non-linear entire $C^2-$solution is given by
\begin{equation*}
u(x,y)=\int{F_1}^{-1}(x)dx +\int{F_2}^{-1}(-y)dy,
\end{equation*}
wherein ${F_i}^{-1}$ is the bijective continuous inverse of $F_i$ ($i = 1,2$).
\end{proof}

\begin{description}
			\item[Example 1.] Taking $F_i(t) = t$   we find that  $u(x,y) = x^2-y^2$ solves the elliptic equation \begin{equation*} u_{xx}+u_{xy}+u_{yy}=0.\end{equation*}
			\item[Example 2.] With $f_i(t)=1+t^2$ \ and  $F_i(t) = t + \frac{t^3}{3}$ we obtain the equation \eqref{toll}.
			\item[Example 3.] Take \ $f_i(t)=\frac{\displaystyle 1}{\displaystyle \sqrt{1+t^2}}$ \ and \ $ F_i(t) =\operatorname{arsinh}t$ \ respectively, \\ \qquad then \ $u(x,y) = \cosh (x) - \cosh (y)$ solves
			\begin{align*}
			&\qquad \frac{u_{xx}}{\sqrt{1+{u_x}^2}}+ \frac{u_{yy}}{\sqrt{1+{u_y}^2}}= 0 \\
			  \text{or} &\qquad \sqrt{1+{u_y}^2}\cdot u_{xx} + u_{xy}+ \sqrt{1+{u_x}^2}\cdot u_{yy} = 0 \\
			  \text{also} &\qquad \sqrt{1+{u_y}^2}\cdot u_{xx}+ 2\tilde{B}u_{xy}+ \sqrt{1+{u_x}^2}\cdot u_{yy} = 0, \\
			 	&\qquad	\text{with an arbitary $\tilde{B}$ such that} \ |\tilde{B}|< \sqrt[4]{(1+{u_y}^2)(1+{u_x}^2)}.
			 \end{align*}
\end{description}

\begin{corollary} With the notation of the theorem we obtain:\\ The function
\begin{equation*}u = u(x,y)=\displaystyle \int{F_1}^{-1}(x)dx +\int{F_2}^{-1}(-y)dy\end{equation*}
also solves the equation
\begin{equation*}\frac{u_{xx}}{f_2(u_y)}+ 2\tilde{B}u_{xy}+ \frac{u_{yy}}{f_1(u_x)} = 0,\end{equation*} with an arbitary $\tilde{B}$, i.e. this equation does not have the Bernstein property in $\mathds R^2$.\\
(For the ellipticity of our last equation assume $|\tilde{B}|< \frac{\displaystyle 1}{\displaystyle \sqrt{f_2(u_y)f_1(u_x)}}$.)
\end{corollary}

\begin{remark}
Both the bijectivity and the strict positivity of $F_i'$ respectively are essential for the conclusion of the theorem.
\end{remark}
 In fact we have the following counterexamples:
\begin{description}
			\item[Example] Take $f_i(t)=\frac{\displaystyle 1}{\displaystyle 1+t^2}$ and $F_i(t)=\arctan(t)$ respectively (so $F_i: \mathds R \to \mathds R$ is not bijective), then we get:\\
			$u(x,y)=\ln(\cos y) - \ln(\cos x)\in C^2\left(\left(-\frac{\pi}{2};\frac{\pi}{2}\right)\times\left(-\frac{\pi}{2};\frac{\pi}{2}\right)\right)$ which solves
			\begin{equation*} \text{the minimal surface equation } (1+{u_y}^2)u_{xx}-2u_x u_y u_{xy}+ (1+{u_x}^2)u_{yy} = 0
			\end{equation*}
		and	the equation  $\frac{\displaystyle u_{xx}}{\displaystyle 1+{u_x}^2}+\frac{\displaystyle u_{yy}}{\displaystyle 1+{u_y}^2}=0$ respectively, clearly not on all of $\mathds R^2$.
\end{description}
 The condition $F_i' > 0$ cannot be replaced by the strong monotonicity of $F_i$ (for $i = 1$ or $i = 2$). Otherwise the solution $u$ can develop singularities:
\begin{description}
			\item[Example] Take $f_i(t)=t^2$ and $F_i(t)=\frac{1}{3}t^3$  respectively, then we obtain:\\
			$u(x,y)=\frac{9}{4}\left(|x|^{4/3}-|y|^{4/3}\right)$ which solves the equation \begin{equation}{u_x}^2u_{xx}+2u_x u_y u_{xy}+ {u_y}^2u_{yy} = 0.\label{aronsson}\end{equation}
			\textsc{Aronsson} presented this $C^{1,1/3}(\mathds R^2)-$ "singular solution" of \eqref{aronsson} in \cite{Aronsson2}:\\  $u$ is $C^1$ in $\mathds R^2$, $C^\infty$ in each open quadrant and the coordinate axes are lines of singularity for $u$.\\ Interestingly, the equation \eqref{aronsson} has the Bernstein-property, see \cite{Aronsson}.
			\end{description}

\begin{acknowledgement}
This paper is a part of my diploma thesis written under supervision of Prof. Ulrich \textsc{Dierkes}.
\end{acknowledgement}

\bigskip
\footnotesize
\begin{AddedInProof}
Quite recently, I have found the presentation of \emph{P. A.} \textsc{Bezborodov} at the International Conference on Analysis and Geometry (1999, Novosibirsk, Russia) where a similar result was stated, however no proofs were given, cf. {\cyr Bezborodov, P.A., Kontrprimer k gipoteze Sa\u{i}mona, Tezisy Trudov Mezhdu\-narodno\u{i} konferentsii po analizu i geo\-metrii, Novosibirsk, 30 avg.-3 sent. 1999. -- Novosibirsk: Izd-vo IM SO RAN, 1999. -- S. 10--11}. \emph{(}\textsc{Bezborodov} \emph{P. A.}, A Counterexample to Simon's Conjecture, \emph{Novosibirsk, 1999}, in Russian.\emph{)}
\end{AddedInProof}

\end{document}